\documentclass[11p,reqno]{amsart}
\usepackage{amsmath}
\usepackage{amsfonts}
\usepackage{amssymb}
\usepackage{graphicx}
\usepackage{color}
\newtheorem{theorem}{Theorem}[section]
\newtheorem*{theorem*}{Theorem}
\newtheorem{lemma}{Lemma}[section]
\newtheorem{corollary}[theorem]{Corollary}
\newtheorem{proposition}{Proposition}[section]

\def\p{\partial}

\def\C{\Bbb C}

\def\aint{\frac{\ \ }{\ \ }{\hskip -0.4cm}\int}

\numberwithin{equation}{section}

\begin{document}
	\title[Gap Theorem]{Gap theorem on K\"ahler manifold with nonnegative orthogonal bisectional curvature}

\author{Lei Ni}
\address{Department of Mathematics, University of California, San Diego, La Jolla, CA 92093, USA}
\email{lni@math.ucsd.edu}

\author{Yanyan Niu}
\address{School of Mathematical Sciences, Capital Normal University, Beijing, China}
\email{yyniukxe@gmail.com}

\thanks{  The research of the  first author is partially supported by NSF grant DMS-1401500.  }

\thanks{  The research of the second author is partially supported by NSFC grant NSFC-11301354, NSFC-11571260, and by Youth Innovation Research Team of Capital Normal University.}

\maketitle

\begin{abstract} In this paper we prove a gap theorem for K\"ahler manifolds with nonnegative orthogonal bisectional curvature and nonnegative Ricci curvature, which generalizes an earlier result of the first author. We also prove a Liouville theorem for plurisubharmonic functions on such a manifolds, which generalizes a previous result of L.-F. Tam and the first author.
\end{abstract}

\section{Introduction}

In \cite{Ni-gap}, the following result was proved.

\begin{theorem}\label{tm:gap}
		Let $(M^n, g)$ be a complete noncompact K\"ahler manifold with nonnegative  bisectional curvature. Then $M$ is flat if for some $o\in M$,
		\begin{equation}\label{eq:1}
		\frac{1}{V_o(r)}\int_{B_o(r)} S(y)d\mu(y)=o(r^{-2}),
		\end{equation}
		where $V_o(r)$ is the volume of $B_o(r)$ and $S(y)$ is the scalar curvature.
		\end{theorem}

 A result of this type was originated by Mok-Siu-Yau in \cite{MSY}, where it was proved that $M$ is isometric to $\C^m$ under much stronger assumptions that $(M^m, g)$ (with $m\ge 2$) is of maximum volume growth (meaning that $V_o(r)\ge \delta r^{2m}$  for some $\delta>0$) and $\mathcal{S}(x)$ decays pointwisely as $r(x)^{-2-\epsilon}$ for some $\epsilon>0$. A Riemannian version of this result in \cite{MSY} was proved  by Greene-Wu  \cite{GW} shortly afterwards (see also \cite{GW2} for related results).
In  \cite{Ni}, Theorem 5.1, with a parabolic method introduced on solving the so-called Poincar\'e-Lelong equation,  the result of \cite{MSY} was improved to the cases covering  manifolds of more general volume growth. Since then there are several further works aiming to prove the optimal result. See for example  \cite{nst}, \cite{CZ}.  In  particular the Ricci flow method was applied in one of these papers. In \cite{NT1}, using a Liouville theorem concerning the plurisubharmonic functions on a complete K\"ahler manifold, and the solution of Poincar\'e-Lelong equation obtained therein, Theorem \ref{tm:gap}  was proved with  an additional exponential growth assumption on the integral of {\it the square of the scalar curvature} over geodesic balls, which was removed in \cite{Ni-gap} using a different method.

The approach of \cite{Ni-gap} toward Theorem \ref{tm:gap} is via the asymptotic behavior of the optimal solution obtained by evolving a $(1,1)$-form with the initial data being the Ricci form through the heat flow of the Hodge-Laplacian operator. The key component of the proof is the monotonicity obtained in \cite{Ni-JAMS} (see also \cite{NN}), which makes the use of the nonnegativity of the bisectional curvature crucially. On the other hand, in \cite{NT2}, the authors proved that the method of deforming a $(1,1)$-form via the Hodge-Laplacian heat equation and studying the asymptotic behavior of the solution can be applied to solve the Poincar\'e-Lelong equation and obtain an optimal solution for it. Namely the following result was proved.

\begin{theorem}\label{tm:NT}
		Let $(M^n, g)$ be a complete noncompact K\"ahler manifold with nonnegative Ricci curvature and nonnegative quadratic orthogonal bisectional curvature. Suppose that $\rho$ is a smooth closed real $(1,  1)$-form on $M$ and let $f=\|\rho\|$ be the norm of $\rho$. Suppose that
		\begin{equation}\label{eq:2}
		\int_0^\infty k_f(r)dr<\infty,
		\end{equation}
		where
		$$
		k_f(r)=\frac{1}{V_o(r)}\int_{B_o(r)}|f|d\mu,
		$$
	    for some fixed point $o\in M$. Then there is a smooth function $u$ so that $\rho=\sqrt{-1}\p\bar{\p}u$. Moreover, for any $0<\epsilon<1$, $u$ satisfies
	    \begin{eqnarray}\label{eq:3}
	    &\,&\alpha_1 r\int_{2r}^\infty k_{\|\rho\|}(s)ds+\beta_1\int_0^{2r}sk_{\|\rho\|}(s)ds\ge u(x)\nonumber\\
	    &\ge &\beta_3\int_0^{2r}s k_{\|\rho\|}(s)ds-\alpha_2 r\int_{2r}^\infty k_{\|\rho\|}(s)ds-\beta_2\int_0^{\epsilon r}s k_{\|\rho\|}(s)ds
	    \end{eqnarray}
	    for some positive constants $\alpha_1(n), \alpha_2(n, \epsilon)$ and $\beta_i(n), 1\le i\le 3$, where $r=r(x)$.
	\end{theorem}

Recall that a K\"ahler manifold $(M^n, g)$ is said to have nonnegative quadratic orthogonal bisectional curvature (NQOB for short) if, at any point $x\in M$ and any unitary frame $\{e_i\}$,
$\sum_{i, j}R_{i\bar{i}j\bar{j}}(a_i-a_j)^2\ge 0$
for all real numbers $a_i$.
$(M^n, g)$ is said to have nonnegative orthogonal bisectional curvature (NOB for short) if for any orthogonal $(1, 0)$ vector fields $X, Y$, $R(X, \bar{X}, Y, \bar{Y})\ge 0$. The example constructed in \cite{HT} shows that the curvature condition (NOB) is stronger than (NQOB). On the other hand,  examples constructed in this paper show that the (NOB) is weaker than the nonnegativity of the bisectional curvature.

A natural question is whether or not the gap theorem remains true under the assumption of Theorem \ref{tm:NT}, or less ambitiously under  the nonnegativity of the orthogonal bisectional curvature, and the nonnegativity of the Ricci curvature.

Related to the gap theorem, a Liouville type theorem was proved in \cite{NT1} for plurisubharmonic functions. \begin{theorem}\label{tm:Liouville}
		Let $M$ be a complete noncompact K\"ahler manifold with nonnegative holomorphic bisectional curvature. Let $u$ be a continuous plurisubharmonic function on $M$. Suppose that
		\begin{equation}\label{eq:4}
		\lim_{x\rightarrow \infty} \frac{u(x)}{\log r(x)}=0.
		\end{equation}
		Then $u$ must be constant.
	\end{theorem}
Very recently, using a partial maximum principle the same Liouville  result was proved \cite{Liu} for complete K\"ahler manifolds with nonnegative holomorphic sectional curvature. On the other hand, there exists an algebraic curvature (\cite{Tam}) which has positive holomorphic sectional curvature, positive orthogonal bisectional curvature (hence positive Ricci curvature), but with negative bisectional curvature for some pair of vectors. This indicate  that the (NOB) condition is in a sense independent to the nonnegativity of the holomorphic sectional curvature. Generalizing the Liouville theorem to manifolds with (NOB) becomes an interesting itself. Note that the solution constructed in Theorem \ref{tm:NT} above for $\rho$ being the Ricci form satisfies the estimate (\ref{eq:4}) if the scalar curvature satisfies the assumption (\ref{eq:1}).
In fact, the assumption (\ref{eq:1}) implies that $r\int_{2r}^\infty k_{\|\rho\|}(s) ds=o(1), \int_0^{2r} sk_{\|\rho\|}(s)ds=o(\log r)$.
Hence to obtain the gap theorem under the weaker assumptions of nonnegative orthogonal bisectional curvature and nonnegative Ricci curvature one only needs to prove the above Liouville theorem for the manifolds with  these weaker assumptions.

The main purpose of this paper is to prove Theorem \ref{tm:gap} and Theorem \ref{tm:Liouville} under the weaker assumptions of  nonnegative orthogonal bisectional curvature and nonnegative Ricci curvature. In fact  the following more general result is proved.

\begin{theorem}\label{thm:main}
Let $(M, g)$ be a complete K\"ahler manifold with nonnegative orthogonal bisectional curvature and nonnegative Ricci curvature. Assume that $\rho\ge 0$ is a smooth $d$-closed $(1,1)$-form.
Suppose that
\begin{equation}\label{ave-decay1}
\int_0^r s\aint_{B_{o}(s)}\|\rho\|(y)\, d\mu(y)\, ds =o(\log r)
\end{equation}
for some $o\in M$. Then $\rho\equiv 0$.
\end{theorem}	

Note that there exists an algebraic curvature (\cite{Tam}) which has positive holomorphic sectional curvature, positive orthogonal bisectional curvature (hence positive Ricci curvature), but with negative bisectional curvature for some pair of vectors.  Hence the Liouville type result can be viewed complementary to the case of \cite{Liu}. In the last part we show that the perturbation technique of Huang-Tam \cite{HT} based on the unitary construction of Wu-Zheng \cite{WZ} can be adapted to construct examples of complete K\"ahler metrics with unitary symmetry such that its curvature has (NOB), but not nonnegative bisectional curvature.

	\section{Proof of the Liouville and the gap theorems}
	
	In \cite{NT1}, the authors proved a Liouville type theorem for plurisubharmonic functions on a K\"ahler manifold with nonnegative holomorpihc bisectional curvature.  A key ingredient of the proof is a  maximum principle for Hemitian symmetric tensor satisfying the Lichnerowicz heat equation.  In fact, the same results still holds on a K\"ahler manifold with the weaker conditions of (NOB) and nonnegativity of the Ricci curvature. We first recall the maximum princple for Hermitian symmetric tensor $\eta(x, t)$ satisfying the Lichnerowicz heat equation
	\begin{equation}\label{eq:L}
	\left(\frac{\p}{\p t}-\Delta\right) \eta_{\alpha\bar{\beta}}=
		R_{\alpha\bar{\beta}\gamma\bar{\delta}}\eta_{\delta\bar{\gamma}}-\frac{1}{2}\left(R_{\alpha\bar{p}}\eta_{p\bar{\beta}}+R_{p\bar{\beta}}\eta_{\alpha\bar{p}}\right).
	\end{equation}
	
	\begin{theorem}\label{tm:MP}
	Let $(M^n, g)$ be a complete noncompact K\"ahler manifold with nonnegative holomorhphic orthogonal bisectional curvature and nonnegative Ricci curvature. Let $\eta(x, t)$ be a Hermitian symmetric $(1, 1)$ tensor satisfying (\ref{eq:L}) on $M\times [0, T]$ with $0<T<\frac{1}{40a}$ such that $\|\eta\|$ satisfies
	\begin{equation}\label{eq:14}
	\int_M \|\eta\|(x, 0)\exp(-ar^2(x))dx<+\infty,
	\end{equation}	
	and
	\begin{equation}\label{eq:15}
	\liminf_{r\rightarrow \infty}\int_0^T \int_{B_o(r)}\|\eta\|^2(x, t)\exp(-ar^2(x)) dx dt<+\infty.
	\end{equation}
	Suppose at $t=0$, $\eta_{\alpha\bar{\beta}}\ge -b g_{\alpha\bar{\beta}}$ for some constant $b\ge 0$. Then there exists $0<T_0<T$ depending only on $T$ and $a$ so that the following are ture:
	
	(i) $\eta_{\alpha\bar{\beta}}(x, t)\ge -bg_{\alpha\bar{\beta}}(x)$ for all $(x, t)\in M\times [0, T_0]$.
	
	(ii) For any $T_0>t'\ge 0$, suppose there is a point $x'$ in $M^n$ and there exist constants $\nu>0$ and $R>0$ such that the sum of the first $k$ eigenvalues $\lambda_1, \cdots, \lambda_k$ of $\eta_{\alpha\bar{\beta}}$ satisfies
	$$
	\lambda_1+\cdots+\lambda_k\ge -kb+\nu k\phi_{x', R}
	$$
	for all $x$ at time $t'$, where $\phi:[0, \infty)\rightarrow [0, 1]$ is a smooth cut-off function such that $\phi\equiv 1$ on $[0, 1]$ and $\phi\equiv 0 $ on $[2, \infty)$, $\phi_{x', R}(x)=\phi(\frac{d(x, x')}{R})$, the eigenvalues of $\eta$ are of accending order. Then for all $t>t'$ and for all $x\in M$, the sum of the first $k$ eigenvalues of $\eta_{\alpha\bar{\beta}}(x, t)$ satisfies
	$$
	\lambda_1+\cdots+\lambda_k\ge -kb+\nu k f_{x', R}(x, t-t'),
	$$
	where $f_{x', R}$ is the solution of
	$(\frac{\p}{\p t}-\Delta) f=-f$ with initial value $\phi_{x', R}(x).$
	\end{theorem}
	\begin{proof}
		The proof is to observe that the argument of Theorem 2.1 in \cite{NT1} only requires the nonnegativity of the orthogonal bisectional curvature and nonnegativity of the Ricci curvature. For the sake of the completeness we include some details of the argument here and pay special attention on the places where the nonnegativity of the orthogonal bisectional curvature is needed. By (\ref{eq:L}), one has
		\begin{eqnarray}
		(\frac{\p}{\p t}-\Delta)\|\eta\|^2
		&=&-\|\eta_{\alpha\bar{\beta}s}\|^2-\|\eta_{\alpha\bar{\beta}\bar{s}}\|^2+2R_{\alpha\bar{\beta}p\bar{q}}\eta_{q\bar{p}}\eta_{\beta\bar{\alpha}}-2R_{\alpha\bar{p}}\eta_{p\bar{\beta}}\eta_{\beta\bar{\alpha}}\nonumber\\
		&\le & -\|\eta_{\alpha\bar{\beta}s}\|^2-\|\eta_{\alpha\bar{\beta}\bar{s}}\|^2,
		\end{eqnarray}
		where we choose $\{e_\alpha\}$ so that $\eta_{\alpha\bar{\beta}}=\lambda_{\alpha}\delta_{\alpha\bar{\beta}}$. Thus
		\begin{eqnarray*}
		2R_{\alpha\bar{\beta}p\bar{q}}\eta_{q\bar{p}}\eta_{\beta\bar{\alpha}}-2R_{\alpha\bar{p}}\eta_{p\bar{\beta}}\eta_{\beta\bar{\alpha}}
		&=&2\left(\sum R_{\alpha\bar{\alpha}\beta\bar{\beta}}\lambda_{\alpha}\lambda_{\beta}-\sum_{\alpha} R_{\alpha\bar{\alpha}}\lambda_{\alpha}^2\right)\\
		&=&-\sum_{\alpha, \gamma} R_{\alpha\bar{\alpha}\beta\bar{\beta}}(\lambda_{\alpha}-\lambda_{\beta})^2\le 0,
	\end{eqnarray*}
provided that $(M, g)$ has nonnegative quadratic orthogonal bisectional curvature (which is a weaker condition than nonnegativity of the orthogonal bisectional curvature).

Combining with the inequality
$$
2|\nabla \|\eta\|^2|\le \|\eta_{\alpha\bar{\beta}}s\|^2+\|\eta_{\alpha\bar{\beta}\bar{s}}\|^2
$$
it implies  (as in \cite{NT1}) that
$$
(\frac{\p}{\p t}-\Delta) \|\eta\|\le 0.
$$
With Lemma 1.2 in \cite{NT1} which holds on the manifold with nonnegative Ricci curvature and (\ref{eq:14}), then
$$
h(x, t)=\int_M H(x, y, t)\|\eta\|(y) dy
$$
is a solution to the heat equation on $M\times [0, \frac{1}{40a}]$ with intial value $\|\eta\|(x)$. With the assumption (\ref{eq:15}) and Theorem 1.2 proved in \cite{NT3}, there exists $0<T_0<T$ such that $\|\eta\|(x, t)\le h(x, t)$ on $M\times [0, T_0]$.

For any $r_2>r_1$, let $A_o(r_1, r_2)$ denote the annulus $B_o(r_2)\setminus B_o(r_1)$. For any $R>0$, let $\sigma_R$ be the cut-off function which is 1 on $A_o(\frac{R}{4}, 4R)$ and 0 outside $A_o(\frac{R}{8}, 8R)$. We define
$$
h_R(x, t)=\int_M H(x, y, t)\sigma_R(y)\|\eta\|(y, 0)dy.
$$
Then $h_R$ satisfies the heat equation with initial value $\sigma_R\|\eta\|$. Moreover, Lemma 2.2 in \cite{NT1} holds when Ricci curvature is nonnegative. That is, there exists $0<T_0<T$  depending only on $a$ such that

(1) there exists a function $\tau=\tau(r)>0$ with $\lim_{r\rightarrow \infty} \tau(r)=0$ such that for all $R\ge \max\{\sqrt{T_0}, 1\}$ and for all $(x, t)\in A_o(\frac{R}{2}, 2R)\times [0, T_0]$,
	$$
	h(x, t)\le h_R(x, t)+\tau(R).
	$$
	
(2) For any $r>0$,
$$
\lim_{R\rightarrow \infty}\sup_{B_o(r)\times [0, T_0]} h_R=0.
$$

By Lemma 1.2 and Corollary 1.1 in \cite{NT1}, we can find a solution $\phi(x, t)$, $(\frac{\p}{\p t}-\Delta)\phi=\phi$ such that $\phi(x, t)\ge \exp(c(r^2(x)+1))$ for some $c>0$ for all $0\le t\le T$.

 As in \cite{NT1} we only need to prove (ii) by assuming (i) since the proof of (i) is similar, but easier. Without the loss of generality, we assume that $t'=0$ and there exist $x'\in M, \nu>0$ and $R_0>0$ such that the first $k$ eigenvalues $\lambda_1, \cdots, \lambda_k$ of $\eta_{\alpha\bar{\beta}}$ satisfy
$$
\lambda_1+\cdots+\lambda_k\ge -kb+\nu k\phi_{x', R_0}
$$
for all $x$ in $M$ at time $t=0$. For simplicity, we assume that $\nu=1$.

Let $\epsilon>0$, for any $R>0$, define $\psi(x, t, \epsilon, R)=-f_{x', R_0}(x, t)+\epsilon \phi(x, t)+h_R(x, t)+\tau(R)+b$ and let $(\eta_R)_{\alpha\bar{\beta}}=\eta_{\alpha\bar{\beta}}+\psi g_{\alpha\bar{\beta}}$. Then at $t=0$, at each point the sum of the first $k$ eigenvalues of $\eta_R$ is positive.  We want to prove that for any $T_0\ge t>0$ and $R>0$, the sum of the first $k$ eigenvalues of $\eta_R$ in $B_o(R)\times [0, T_0]$ is positive, provided $R$ is large enough.

Then one can argue by contradiction similarly as in the proof of Theorem 2.1 in \cite{NT1}. The only attention to pay is the proof of (2.14) in \cite{NT1}. That is, if $\eta$ has eigenvectors $v_p=\frac{\p}{\p z^p}$ for $1\le p\le n$ with eigenvalue $\lambda_p$,
\begin{eqnarray*}
&\,&\sum_{\alpha,\beta=1}^{k}[ R_{\delta\bar{\gamma}\alpha\bar{\beta}}
(\eta_{\gamma\bar{\delta}}+\psi g_{\gamma\bar{\delta}})-\frac{1}{2}
R_{\alpha\bar{p}}(\eta_{p\bar{\beta}}+\psi g_{p\bar{\beta}})\\
&\qquad&-\frac{1}{2}R_{p\bar{\beta}}(\eta_{\alpha\bar{p}}+\psi g_{\alpha\bar{p}})g^{\alpha\bar{\beta}}]\\
&=&\sum_{\alpha=1}^k\sum_{\gamma=1}^{n}R_{\gamma\bar{\gamma}\alpha\bar{\alpha}}\lambda_{\gamma}-\sum_{\alpha=1}^{k}R_{\alpha\bar{\alpha}}\lambda_{\alpha}\\
&=& \sum_{\alpha=1}^{k}\sum_{\gamma=1}^{n}R_{\gamma\bar{\gamma}\alpha\bar{\alpha}}\lambda_{\gamma}-\sum_{\alpha=1}^{k}\sum_{\gamma=1}^{n}R_{\gamma\bar{\gamma}\alpha\bar{\alpha}}\lambda_{\alpha}\\
&=&\sum_{\alpha=1}^k\sum_{\gamma=k+1}^{n}R_{\gamma
	\bar{\gamma}\alpha\bar{\alpha}}\lambda_{\gamma}-\sum_{\alpha=1}^k\sum_{\gamma=k+1}^{n}R_{\gamma\bar{\gamma}\alpha\bar{\alpha}}\lambda_{\alpha}\\
&=&\sum_{\alpha=1}^{k}\sum_{\gamma=k+1}^{n}R_{\gamma\bar{\gamma}\alpha\bar{\alpha}}(\lambda_{\gamma}-\lambda_{\alpha})\\
&\ge & 0
\end{eqnarray*}
where we only used the fact that $M$ has nonnegative orthogonal bisectional curvature and $\lambda_{\gamma}\ge \lambda_{\alpha}$ for $\gamma\ge \alpha$.

The rest of the proof is the same as that of Theorem 2.1 in \cite{NT1}.  \end{proof}

The next observation is that one still has the following corollary,  as Corollary 2.1 in \cite{NT1}, under the weaker condition (NOB).

\begin{corollary}\label{cor:MP}
	Let $M$ and $\eta$ be as in the Theorem \ref{tm:MP} with $b=0$. That is, $\eta(x, 0)\ge 0$ for all $x\in M$. Let $T_0$ be such that the conclusions of the theorem are true. For $0<t<T_0$, let
	$$
	\mathcal{K}(x, t)=\{w\in T_x^{1, 0}(M)|\, \eta_{\alpha\bar{\beta}}(x, t)w^\alpha=0,\,\mbox{for all}\, \beta\}
	$$
	be the null space of $\eta_{\alpha\bar{\beta}}$. Then there exists $0<T_1<T_0$ such that for any
	$0<t<T_1$, $\mathcal{K}(x, t)$ is a smooth  distribution on $M$.
\end{corollary}

\begin{proposition}\label{prop:Mp2} Let $\eta(x, t)$ be a Hermitian symmetric tensor satisfying (\ref{eq:L}). Assume that $\eta(x, t)\ge 0$ on $M\times (0, T)$ and $M$ has nonnegative orthogonal
bisectional curvature. Then $\mathcal{K}(x, t)$ is invariant under parallel translation. In particular, if $M$ is simply-connected, there is a splitting $M=M_1\times M_2$ with $\eta$ being zero on $M_1$ and positive on $M_2$, and $M_i$ has nonnegative orthogonal bisectional curvature.
\end{proposition}
One can modify the original argument in \cite{NT1} for this slightly more general result. On the other hand the strong maximum principle of Bony adapted by Brendle-Schoen \cite{BS} (cf. Bony \cite{Bo}) can be applied to obtain this result. In this case one formulates everything on the principle $\mathsf{U}(n)$-bundle. For any unitary frame $\mathfrak{e}=\{e_i\}_{i=1}^n$ define $u(\mathfrak{e})=\eta(e_1, e_{\bar{1}})$. Let $\widetilde{Y}$ be the horizontal lifting of the vector field $\frac{\partial}{\partial t}$ on $M\times (0, T)$. At $\mathfrak{e}$, let $\widetilde{X_i}$ be the horizontal lifting of $e_i$. Similar  computation as in \cite {BS} yields that  that
$$
\left(\widetilde{Y}-\sum \widetilde{X}_i\widetilde{X}_i\right)u=\sum_{s\ge 2} R_{1\bar{1}s\bar{s}} (\eta_{s\bar{s}}-\eta_{1\bar{1}})\ge -Ku,
$$
where $K$ is a local constant depending only on $M$. Here we have used that $R_{1\bar{1}s\bar{s}} \eta_{s\bar{s}}\ge 0$ for $s\ge 2$. This is enough to conclude that $\mathcal{K}(x, t)$ is invariant under the parallel translation. In fact the following slight general version of Bony's strong maximum principle holds.
\begin{theorem}[Bony, Brendle-Schoen] Let $\Omega$ be an open subset of $\mathbb{R}^n$. Let $\{X_i\}_{i=1}^{m}$ be smooth vector fields on $\Omega$. Assume that $u: \Omega\to \mathbb{R}$ is a nonnegative smooth function satisfying
$$
\sum_{i=1}^m (D^2u)(X_i, X_i)\le -K \min \{ 0, \inf_{|\xi|=1}(D^2 u)(\xi, \xi)\}+K|\nabla u|+Ku,
$$
with $K$ be a positive constant. Let $Z=\{x\, |\, u(x)=0\}$ be the zero set. Let $\gamma(s): [0, 1]\rightarrow \Omega$ be a smooth curve such that $\gamma(0)\in Z$ and $\gamma'(s)=\sum a_i(s) X_i(\gamma(s))$ with $a_i(s)$ being smooth functions. Then $\gamma(s)\in Z$ for all $s\in [0,1]$.
\end{theorem}
Since $K(x, t)$ is invariant under the parallel translation and clearly it is a clear subspace, the decomposition follows from the De Rham decomposition theorem.

	{\it Proof of Theorem \ref{tm:Liouville} under (NOB) and the nonnegativity of the Ricci curvature.}
		The proof is similar to that of Theorem 3.2 in \cite{NT1}.
		
		Without loss of generality, we may assume that $M$ is simply connected (by lifting the function to the universal cover, the growth condition clearly is preserved). For any fixed constant $c,$ we let $u_c=\max\{u, c\}$. It is well-known that $u_c$ is plurisubharmonic and $u_c$ satisfies
		\begin{equation}\label{eq:13}
		|u_c(x)|\le C\exp(ar^2(x)),
		\end{equation}
		for some constant $C>0$ and $a>0$. By adding a constant, we can also assume that $u_c\ge 0$. Then $u_c$ is a nonnegative continuous plurisubharmonic function satisfying (\ref{eq:13}).
		Now we consider the heat equation
		\[\begin{cases}
			&(\frac{\p}{\p t}-\Delta)v_c(x, t)=0,\\
			& \qquad v_c(x, 0)=u_c(x).
		\end{cases}\]
		Then the above Dirichlet boundary problem has a solution $v_c(x, t)$ on $M\times [0, \frac{1}{40a}]$, obtained in Lemma 1.2 in \cite{NT1} which holds on the manifold with nonnegative Ricci curvature.
		
	  Since Theorem \ref{tm:MP} and Corollary \ref{cor:MP} (and Proposition  \ref{prop:Mp2}) hold on the manifold with (NOB) and nonnegative Ricci curvature, one can go through the proof of Theorem 3.1 in \cite{NT1}, to see that there exists $0<T_0<T$ such that $v_c(x, t)$ is a smooth plursubharmonic function on $M\times (0, T_0]$.
      Moreover there exists $0<T_1<T_0$, such that the null space of $(v_c)_{\alpha\bar{\beta}}(x, t)$
%
		$$
			\mathcal{K}(x, t)=\{w\in T_x^{1, 0}(M)| (v_c)_{ \alpha\bar{\beta}}w^\alpha=0,\,\mbox{for all} \quad \beta\}
			$$
		 is a distribution on $M$ for any $0<t<T_1$. Moreover, the distribution is invariant under parallel translation. Then by Proposition \ref{prop:Mp2},
for any $t_0>0$ small enough, $M=M_1\times M_2$ isometrically and holomorphically such that  when restricted on $M_1$, $(v_c)_{\alpha\bar{\beta}}$ is zero,  and $(v_c)_{\alpha\bar{\beta}}$ is positive everywhere when restricted on $M_2$ by the De Rham decomposition. We want to conclude that $M_2$ factor does not exist. By Corollary 1.1 in \cite{NT1}, we have
		\begin{equation}\label{eq:5}
		\limsup_{x\rightarrow \infty} \frac{v_c(x, t_0)}{\log r(x)} =0.
		\end{equation}
		Hence when restricted on $M_2$ (if the factor exists), (\ref{eq:5}) still holds. This contradicts with the fact that $(v_c)_{\alpha\bar{\beta}}$ is positive when restriced on $M_2$,  since by  Propostion 4.1 of \cite{Ni}, which asserts  that if a plurisubharmonic function $p(x)$ on a K\"ahler manifold with nonnegative Ricci curvature  satisfies the growth condition (\ref{eq:5}), then $(\p\bar{\p}p)^n=0$, where $n$ is the complex-dimension of the manifold. Hence $(v_c)_{k\bar{l}}(x, t_0)\equiv 0$ on $M$ for all $t_0$ small enough. By the gradient estimate of Cheng-Yau \cite{CY} and (\ref{eq:5}), we can conclude that $v_c(x, t_0)$ is a constant, provided $t_0$ is small enough. Hence $u_c$ is a constant. Since $c$ is arbitrary, it shows that $u(x)$ is also a constant.

{\it Proof of Theorem \ref{tm:gap} under the assumptions of (NOB) and nonnegativity of the Ricci curvature. } Let $\rho$ be the Ricci form, which is a smooth nonnnegative closed real $(1, 1)$-form on $M$. It is easy to check that for any $y\in M$,
$$
\|\rho\|(y)\le S(y)\le \sqrt{n}\|\rho\|(y).
$$
Then
\begin{equation}\label{eq:12}
k_{\|\rho\|}(r)=o(r^{-2})
\end{equation}
and
(\ref{eq:2}) follows when (\ref{eq:1}) holds for some fixed point $o\in M$. Since the curvature condition (NOB) is stronger than (NQOB),
  Theorem \ref{tm:NT} still holds with the assumptions of (NOB) and nonnegativity of the Ricci curvature. Moreover, the solution $u$ to the Poincar\'e-Lelong equation $\rho=\sqrt{-1}\p\bar{\p}{u}$ is a plurisubharmonic function and (\ref{eq:3}) holds.
  In fact, (\ref{eq:3}) implies that

   \begin{equation}
   \lim_{x\rightarrow \infty} \frac{u(x)}{\log r(x)}=0，
   \end{equation}
since (\ref{eq:12}) implies that 
$$\int_{2r}^\infty k_{\|\rho\|}(s) ds=o(r^{-1}),\quad \int_0^{2r} s k_{\|\rho\|}(s) ds=o(\log r), \, \mbox{ and } \int_0^{\epsilon r} s k_{\|\rho\|}(s) ds=o(\log r).$$

   By the generalization of the Liuoville theorem proved above, we conclude that $u$ must be constant. This implies that $Ric\equiv 0$. For any unitary frame $\{e_\alpha\}_{\alpha=1}^n$, (NOB) implies that for any $\alpha\neq \beta$, by considering  $\tilde{e}_\alpha=\frac{1}{2}(e_{\alpha}-e_\beta),\, \tilde{e}_\beta=\frac{1}{2}(e_{\alpha}-e_\beta)$, 
  $$
  R(\tilde{e}_\alpha, \bar{\tilde{e}}_\alpha, \tilde{e}_\beta,\bar{\tilde{e}}_\beta)\ge 0,
  $$
  which is equivalent to that
  \begin{equation}\label{eq:6}
  R_{\alpha\bar{\alpha}\alpha\bar{\alpha}}+R_{\beta\bar{\beta}\beta\bar{\beta}}-R_{\alpha\bar{\beta}\alpha\bar{\beta}}-R_{\beta\bar{\alpha}\beta\bar{\alpha}}\ge 0.
  \end{equation}
  If we replace $e_\beta$ by $\sqrt{-1}e_\beta$ in $\tilde{e}_\alpha,$ and $ \tilde{e}_\beta$,
  \begin{equation}\label{eq:7}
  R_{\alpha\bar{\alpha}\alpha\bar{\alpha}}+R_{\beta\bar{\beta}\beta\bar{\beta}}+R_{\alpha\bar{\beta}\alpha\bar{\beta}}+R_{\beta\bar{\alpha}\beta\bar{\alpha}}\ge 0.
  \end{equation}

  By summing (\ref{eq:6}) and (\ref{eq:7}), we  obtain the following inequality
  $$
  R_{\alpha\bar{\alpha}\alpha\bar{\alpha}}+R_{\beta\bar{\beta}\beta\bar{\beta}}\ge 0.
  $$
  But
  \begin{eqnarray*}
  	 R_{\alpha\bar{\alpha}\beta\bar{\beta}}&\ge&  0,\\
R_{\alpha\bar{\alpha}}&=&R_{\alpha\bar{\alpha}\alpha\bar{\alpha}}+\sum_{\gamma\neq \alpha} R_{\alpha\bar{\alpha}\gamma\bar{\gamma}}=0,\\
R_{\beta\bar{\beta}}&=&R_{\beta\bar{\beta}\beta\bar{\beta}}+\sum_{\gamma\neq \beta} R_{\beta\bar{\beta}\gamma\bar{\gamma}}=0,\\
  \end{eqnarray*}
  Then $R_{\alpha\bar{\alpha}\alpha\bar{\alpha}}=0$ for any $\alpha$, which implies that $M$ is flat and the generalization of Theorem \ref{tm:gap} follows.

{\it Proof of Theorem \ref{thm:main}}. By Theorem \ref{tm:NT}, Theorem \ref{tm:Liouville}, it suffices to establish the estimate:
\begin{equation}\label{eq:help11}
r\int_{2r}^\infty k_{\|\rho\|}(s)\, ds =o(\log r).
\end{equation}
From (\ref{ave-decay1}), we know $\int_{\frac{1}{2}r}^{r} sk_{\|\rho\|}(s) ds=o(\log r)$. For any $\frac{1}{2}r\le s\le r$, by volume comparison, 
\begin{eqnarray}
&\,& 2^{-2n} k_{\|\rho\|}(\frac{r}{2})\le
\frac{Vol(B_o(\frac{r}{2}))}{Vol(B_o(r))}k_{\|\rho\|}(\frac{2}{r})\nonumber\\
&\le&
k_{\|\rho\|}(s)=\frac{1}{Vol(B_o(s))}\int_{B_o(s)}\|\rho\|(y)d\mu(y)\nonumber\\
&\le& \frac{Vol(B_o(r))}{Vol(B_o(\frac{r}{2}))}k_{\|\rho\|}(r)\le 2^{2n}k_{\|\rho\|}(r).\nonumber
\end{eqnarray}
From this and $\int_{\frac{r}{2}}^{r} sk_{\|\rho\|}(s) ds=o(\log r)$, we derive 
$$
k_{\|\rho\|}(r)=o\left(\frac{\log r}{r^2}\right).
$$
This implies (\ref{eq:help11}), for
\begin{equation}
\lim_{r\rightarrow \infty} \frac{\int_{2r}^\infty k_{\|\rho\|}(s) d s}{r^{-1}\log r}
=2\lim_{r\rightarrow \infty} \frac{r^2 k_{\|\rho\|}(2r)}{\log r-1}=0.\nonumber
\end{equation}

  \section{Examples}

In \cite{WZ}, Wu and Zheng  considered the $\mathsf{U}(n)$-invariant K\"ahler metrics on $\mathbb{C}^n$ and obtained necessary and sufficient conditions for the nonnegativity of the curvature operator, nonnegativity of the  sectional curvature, as well as the  nonnegativity of the  bisectional curvature respectively.
  In \cite{YZ}, Yang and Zheng later proved that the necessary and sufficient condition in \cite{WZ} for the nonnegativity of the  sectional curvature  holds for the nonnegativity of the complex sectional curvature under the unitary symmetry. In \cite{HT}, the authors  obtained the necessary and sufficient conditions for (NOB) and (NQOB) respectively. Moreover  they constructed a $\mathsf{U}(n)$-invariant K\"ahler metric on $\mathbb{C}^n$,  which is of  (NQOB), but does not have (NOB) nor nonnegativity of the Ricci curvature. In this section, we will construct a $\mathsf{U}(n)$-invariant K\"ahler metric on $\mathbb{C}^n$ which has (NOB) but does not have nonnegative bisectional curvature. The existence of such metric was pointed out in Remark 4.1 of \cite{HT}, and the construction below is a modification of the perturbation construction therein.

  We follow the same notations as in \cite{WZ, YZ}. 
  Let $(z_1, \cdots, z_n)$ be the standard coordinate on $\mathbb{C}^n$ and $r=|z|^2$. A $\mathsf{U}(n)$-invariant metric on $\mathbb{C}^n$ has the K\"ahler form
  \begin{equation}
  \omega=\frac{\sqrt{-1}}{2}\p\bar{\p} P(r)
  \end{equation}
  where $P\in C^\infty\left([0, +\infty)\right)$. Under the local coordinates, the metric has the components:
  \begin{equation}\label{eq:g}
  g_{i\bar{j}}=f(r)\delta_{ij}+f'(r)\bar{z}_i z_j.
  \end{equation}
  We further denote:
  \begin{equation}\label{eq:g1}
  f(r)=P'(r), \quad h(r)=(rf)'.
  \end{equation}
It is easy to check that $\omega$ will give a complete K\"ahler metric on $\mathbb{C}^n$ if and only if
  \begin{equation}\label{eq:10}
  f>0, \, h>0, \, \int_0^\infty \frac{\sqrt{h}}{\sqrt{r}}dr=+\infty.
  \end{equation}
  If $h>0$, then $\xi=-\frac{rh'}{h}$ is a smooth function on $[0, \infty)$ with $\xi(0)=0$. On the other hand, if $\xi$ is a smooth function on $[0, \infty)$ with $\xi(0)=0$, one can  define $h(r)=\exp(-\int_0^r \frac{\xi(s)}{s}ds)$ and $f(r)=\frac{1}{r}\int_0^r h(s)$ ds with $h(0)=1$. It is easy to see that $\xi(r)=-\frac{rh'}{h}$.
  Then (\ref{eq:g}) defines a $\mathsf{U}(n)$-invariant K\"ahler metric on $\mathbb{C}^n$.

   The components of the curvature operator of a $\mathsf{U}(n)$-invariant K\"ahler metric under the orthonormal frame $\{e_1=\frac{1}{\sqrt{h}}\p_{z_1}, e_2=\frac{1}{\sqrt{f}}\p_{z_2}, \cdots, e_n=\frac{1}{\sqrt{f}} \p_{z_n}\}$ at $(z_1, 0, \cdots, 0)$  are given as follows, see \cite{WZ}:
  \begin{eqnarray}
  	A&=& R_{1\bar{1}1\bar{1}}=-\frac{1}{h}\left(\frac{rh'}{h}\right)'=\frac{\xi'}{h}; \label{eq:A}\\
  	B&=& R_{1\bar{1}i\bar{i}}=\frac{f'}{f^2}-\frac{h'}{hf}=\frac{1}{(rf)^2}\left[rh-(1-\xi)\int_0^r h(s)\,  ds\right],\,  i\ge 2;\label{eq:B}\\
  	C&=& R_{i\bar{i}i\bar{i}}=2R_{i\bar{i}j\bar{j}}=-\frac{2f'}{f^2}=\frac{2}{(rf)^2}\left(\int_0^r h(s)\,  ds-rh\right),\, i\neq j, i, j\ge 2.\label{eq:C}
  \end{eqnarray}
   The other components of the curvature tensor are zero, except those obtained by the symmetric properties of curvature tensor.

   The following result was  proved in \cite{WZ}, which plays an important role in the construction.
   \begin{theorem}[Wu-Zheng]
   	(1) If $0<\xi<1$ on $(0, \infty)$, then $g$ is complete.
   	
   	(2) $g$ is complete and has positive  bisectional curvature if and only if $\xi'>0$ and $0<\xi<1$ on $(0, \infty)$, where $\xi'>0$ is equivalent to $A>0, B>0$ and $C>0$.
   	
   	(3) Every complete $\mathsf{U}(n)$-invariant K\"ahler metric on $\mathbb{C}^n$ with positive bisectional curvature is given by a smooth function $\xi$ in (2).
   \end{theorem}

  Using the above notations and formulations, in \cite{HT},  the authors proved the following theorem.
  \begin{theorem}[Huang-Tam]\label{tm:HT}
  	An $\mathsf{U}(n)$-invariant K\"ahler metric on $\mathbb{C}^n$ has nononegative orthogonal bisectional curvature if and only if $A+C\ge 0, B\ge 0$ and $C\ge 0$.
  \end{theorem}

 Let $\xi$ be a smooth function on $[0, \infty)$ with $\xi(0)=0, \xi'(r)>0$ and $0<\xi(r)<1$ for $0<r<\infty$. Let $a=\lim_{r\rightarrow \infty} \xi(r)$. Then $0<a\le 1$. By the above this gives a complete $\mathsf{U}(n)$-invariant metric on $\mathbb{C}^n$ with positive bisectional curvature. The strategy of \cite{HT} is to perturb this metric by adding a perturbation term to $\xi$ to obtain the one with needed property.  In particular, \cite{HT}, produces metric with (NQOB), but does not satisfy (NOB) nor nonnegativity of the Ricci curvature. For that  \cite{HT}  first obtained the following estimates \cite{HT,WZ} (cf. Lemma 4.1 of \cite{HT}).
 \begin{lemma}\label{lm:HT}
 	Let $\xi$ be as above with $\lim_{r\to \infty}\xi =a \, (\in (0, 1))$. We have the following:
 	
(1) For $r>0$, $\left( r h-(1-\xi)\int_0^r h\right)' >0$, and
 $$
 \lim_{r\to \infty} \int_0^r h=\infty, \quad  \lim_{r\to \infty} h=0, \quad \lim_{r\to \infty} \frac{rh}{\int_0^r h}=1-a.
 $$

(2) For any $\epsilon>0$, and for any $r_0>0$, there is $R>r_0$ such that
 	$$
 	\xi'(R)-\epsilon h(R)C(R)<0.
 	$$
 	
(3) $\lim_{r\rightarrow \infty} h(r)C(r)=0$.
 	
(4) For all $\epsilon>0$, there exists $\delta>0$ such that if $R\ge 3$, $\delta\ge \eta\ge 0$ is a smooth function with support in $[R-1, R+1]$, then for all $r\ge 0$,
 	$$
 	h(r)\le \bar{h}(r)\le (1+\epsilon)h(r), \quad\mbox{and}\, \int_0^r h\le \int_0^r \bar{h}\le (1+\epsilon)\int_0^r h,
 	$$
 	where $\bar{h}(r)=\exp(-\int_0^r \frac{\bar{\xi}}{t}dt)$ and $\bar{\xi}=\xi-\eta$.
 \end{lemma}

 Let $\phi$ be a cutoff function on $\mathbb{R}$ as in \cite{HT} such that

(i) $0\le \phi\le c_0$ with $c_0$ being an absolute constant;

(ii) $\mbox{supp} (\phi) \subset [-1, 1]$;

(iii) $\phi'(0)=1$ and $|\phi'|\le 1$.

The construction is to perturb $\xi$ into $\bar{\xi}(r)=\xi(r)-\alpha h(R)C(R)\phi(r-R)$ for suitable choice of $R$, $\alpha$. Note that this only changes the value of $\xi$ on a compact set. Once $\bar{h}$ is defined, equations (\ref{eq:A})--(\ref{eq:C}) define the corresponding curvature
components $\bar{A}, \bar{B}, \bar{C}$ of the perturbed metric.
\begin{theorem}
	There is $1>\alpha>0$ such that for any $r_0>0$ there is $R>r_0$ satisfying the following: If $\bar{\xi}(r)=\xi(r)-\alpha h(R)C(R)\phi(r-R)$, then $\bar{\xi}$ determines a complete $\mathsf{U}(n)$-invariant K\"ahler metric on $\mathbb{C}^n$ such that
	
	\begin{enumerate}
		\item $\bar{A}(R)<0$;
		\item $\bar{A}+\bar{C}>0$ on $[R-1, R+1]$;
		\item $\bar{B}(r)>0$ for all $r$; and
		\item $\bar{C}(r)>0$ for all $r$.
			\end{enumerate}
		Then $\bar{\xi}$ will give a compete $\mathsf{U}(n)$-invariant K\"ahler metric which satisfies (NOB) but does not have nonnegative bisectional curvature.
\end{theorem}

\begin{proof}
	Let $\beta=\alpha h(R)C(R)$. Let $\epsilon>0$ and $R\ge 3$ to be chosen later. Then by Lemma \ref{lm:HT} (3), there is $a>\delta>0$ independent of $R$ such that $\bar{\xi}(r)=\xi(r)-\beta\phi(r-R)$ with $\beta>0$ and $\beta c_0<\delta$, then $\bar{\xi}$ determines a complete $\mathsf{U}(n)$-invariant K\"ahler metric on $\mathbb{C}^n$ such that for all $r$
	\begin{equation}
	h(r)\le \bar{h}(r)\le (1+\epsilon)h(r), \, \int_0^r h\le \int_0^r \bar{h}\le (1+\epsilon)\int_0^r h.
	\end{equation}
Hence $\bar{\xi}$, $\bar{h}$, $\bar{f}$ will define a complete unitary symmetric K\"ahler metric on $\mathbb{C}^n$. Assuming (1)-(4) in the theorem, (1)  implies that the metric does not have nonnegative holomorphic sectional curvature (hence can not have non-negative bisectional curvature). The equation (\ref{eq:A}) together with (4) implies that (2) is sufficient to conclude that $\bar{A}+\bar{C}>0$ for all $r$.  Hence by Theorem \ref{tm:HT} the perturbed metric has (NOB).
	
	By Lemma \ref{lm:HT} (2), there is $R$ large enough, such that
	$$
	\bar{\xi}'(R)=\xi'(R)-\beta\le (\epsilon-\alpha)h(R)C(R).
	$$
Hence for (1), it suffices to choose $\alpha >\epsilon$.

 By the formula (\ref{eq:C}) and the proof of Lemma 4.2 in \cite{HT}, we may choose a large $r_1$ so that if $R>r_1$ and for $r\in [R-1, R+1]$,
	$$
	\bar{C}(r)\ge \frac{2}{(1+\epsilon)^2\int_0^R h}(a-2\epsilon+a\epsilon-\epsilon^2)
	$$
	provided $a-2\epsilon+a\epsilon-\epsilon^2>0$. We choose $\epsilon>0$ so that it satisfies this condition. On the other hand,
	$$
	C(R)\le \frac{2}{\int_0^R h}(a+\epsilon)
	$$
	if $r_1$ is large enough depending only on $\epsilon$ and $R>r_1$. Hence, if $\epsilon$ and $r_1$ satisfy the above conditions, then for $r\in [R-1, R+1]$,
	$$
	\bar{C}(r)\ge \frac{a-2\epsilon+a\epsilon-\epsilon^2}{(a+\epsilon)(1+\epsilon)^2}C(R).
	$$
	Therefore, if $\epsilon>0$ satisfies $a>\epsilon$ and $a-2\epsilon+a\epsilon-\epsilon^2>0$, then we can find $r_1>r_0$ such that if $R>r_1$, then for $r\in  [R-1, R+1]$,
	\begin{eqnarray}
	\bar{A}(r)+\bar{C}(r)&\ge& \frac{\xi'(r)-\beta}{\bar{h}}+\bar{C}(r)\nonumber\\
	&\ge& \frac{-\beta}{\bar{h(r)}}+ \frac{a-2\epsilon+a\epsilon-\epsilon^2}{(a+\epsilon)(1+\epsilon)^2}C(R)\\
	&\ge & -\frac{\beta}{(1-\epsilon)h(R)}+\frac{a-2\epsilon+a\epsilon-\epsilon^2}{(a+\epsilon)(1+\epsilon)^2}C(R)\nonumber\\
	&= &\frac{1}{(1-\epsilon)h(R)}[-\beta+(1-\epsilon)\frac{a-2\epsilon+a\epsilon-\epsilon^2}{(a+\epsilon)(1+\epsilon)^2}h(R)C(R)]\nonumber.
	\end{eqnarray}
	We can choose $\alpha$ with $0<\epsilon<\alpha<1$,  which is a fixed constant depending only on $\epsilon, a$ and $n$ such that
	$$
	\epsilon<\alpha<(1-\epsilon)\frac{a-2\epsilon+a\epsilon-\epsilon^2}{(a+\epsilon)(1+\epsilon)^2}	$$
	Then $\bar{A}(r)+\bar{C}(r)>0$ on $[R-1, R+1]$. To achieve all the requirement above  we can pick $\epsilon$ sufficiently small,  $a=\frac{1}{2}$ and $\epsilon<\alpha<1$. Hence (1) and (2) follows.
	
	To prove (3),  we can appeal Lemma 4.3 (1) of \cite{HT}, for $r_1 $ is large enough and $R>r_1$. In fact by the formula (\ref{eq:C}) and Lemma \ref{lm:HT} it can be seen that
\begin{equation}\label{eq:hp1}
\lim_{r\to \infty} C(r)\int_0^r h =2\left(1-\frac{rh}{\int_0^r h}\right)=2a
\end{equation}
which implies that $\beta(r)=\alpha C(r) h(r) $ satisfies that $\beta(r)\int_0^r h\to 0$ as $r\to \infty$.
On the other hand, for  any $\epsilon_1>0$, there exists $\delta_1>0$ such that if $\alpha c_0\le 2\delta_1$ and $R$ sufficiently large the conclusion in (4) of Lemma 3.1 holds with $\epsilon$ replaced with $\epsilon_1>0$.  The computation in the proof of Lemma 4.3 shows that
\begin{equation}\label{eq:hp2}
(r\bar{f})^2 \bar{B}(r) = r\bar{h}-(1-\bar{\xi}(r))\int_0^r\bar{h} =\int_0^r (\bar{\xi}(r)-\bar{\xi}(t))\bar{h}(t)\, dt.
\end{equation}
Using $\xi'>0$ and $h'<0$ the above gives
$$
(r\bar{f})^2 \bar{B}(r)\ge \int_0^r (\xi(r)-\xi(t))h(t)\, dt-2\delta_1(1+\epsilon_1)h(R-1)-c_0\beta(R)\int_0^R h(t).
$$
Using the second part of (\ref{eq:hp2}) again we have that
$$
\int_0^r (\xi(r)-\xi(t))h(t)\, dt=rh(r) -(1-\xi(r))\int_0^r h(t)\, dt \ge \eta>0
$$
for some $\eta$ by Lemma \ref{lm:HT} part (1). 	This shows that $\bar{B}(r)>0$ if $R\ge r_1$ for some $r_1$ large.

	Since $\bar{h}'=-\frac{\bar{h}\bar{\xi}}{r}<0$ when $r>0$, then $\int_0^r \bar{h}\ge \bar{h} r$ when $r>0$, which implies that (4) satisfies, by the formula (\ref{eq:B}). This provides a simplification of the proof of Lemma 4.3 part (ii) in of \cite{HT}.
\end{proof}

The computation of \cite{HT} also implies the following result.

  \begin{theorem}\label{tm:OBC}
  	A $\mathsf{U}(n)$-invariant K\"ahler metric on $\mathbb{C}^n$ has nononegative orthogonal bisectional curvature and nonnegative Ricci curvature if and only if $A+C\ge 0$,$A+(n-1)B\ge 0$, $ B\ge 0$ and $C\ge 0$.
  \end{theorem}

In fact, algebraically  one can construct a curvature of  (NOB) and nonnegative Ricci curvature but does not have nonnegative holomorphic bisectional curvature.
By Theorem \ref{tm:OBC}, to construct an example of unitary symmetry as the above with  (NOB) and nonnegative Ricci curvature, but not nonnegative bisectional curvature it suffices to find a positive smooth function  on $[0, \infty)$ with $\xi:(0, \infty)\to (0, 1),\, \xi(0)=0$ which satisfies $\xi'<0$ somewhere, but 
\begin{eqnarray*}
B\ge 0&\Longleftrightarrow& rh-(1-\xi)\int_0^r hds\ge 0;\\
C\ge 0&\Longleftrightarrow& \int_0^r h(s)\, ds-rh(r)\ge 0;\\
A+C\ge 0&\Longleftrightarrow&  \xi'+\frac{2h}{(rf)^2}\left(\int_0^r h(s)ds-rh(r)\right)\ge 0;\\
A+(n-1)B\ge 0&\Longleftrightarrow& \xi'+\frac{(n-1)h}{(rf)^2}\left(rh-(1-\xi)\int_0^r h(s)\, ds\right)\ge 0.
\end{eqnarray*}
Note that the second condition always holds as long as $\xi>0$. We are not completely sure if a similar  construction would yield a such metric.


\section*{Acknowledgments.} { We would like to thank Professor Luen-Fei Tam for suggesting the  problem of the gap theorem for manifolds with nonnegative orthogonal bisectional curvature.  }

\end{document}